\documentclass[a4paper, 11pt, oneside, leqno]{scrartcl}
\usepackage[T1]{fontenc}
\usepackage[utf8]{inputenc}

\usepackage{authblk}

\usepackage{amsthm}
\usepackage{amsmath}
\allowdisplaybreaks
\usepackage{amsfonts}
\usepackage{amssymb}
\usepackage{graphicx}

\usepackage{amsxtra}
\usepackage{enumitem}

\usepackage{ifthen}  
\usepackage{mathrsfs} 

\usepackage[left=3cm, right=3cm, top=3cm, bottom=2.5cm]{geometry}

\usepackage[nodisplayskipstretch]{setspace}
\usepackage[leqno]{mathtools} 
\usepackage{xcolor}

\RedeclareSectionCommand[
  runin=true,
	]{subsection}
	\RedeclareSectionCommand[
  runin=true,
	]{subsubsection}

\usepackage[square,sort,comma,numbers]{natbib} 
\usepackage{cite} 

\usepackage{hyperref}
\hypersetup{colorlinks=true,linkcolor=blue,citecolor=blue,urlcolor=blue,filecolor=blue}
\usepackage{cleveref} 

\newtheorem{theorem}{Theorem}[section]
\theoremstyle{plain}

\newtheorem{lemma}[theorem]{Lemma}

\theoremstyle{definition}
\newtheorem{definition}[theorem]{Definition}
\newtheorem{problem}[theorem]{Problem}

\theoremstyle{remark}
\newtheorem{remark}[theorem]{Remark}

\newtheorem{notation}[theorem]{Notation}

\newtheorem*{ack*}{Acknowledgements}

\numberwithin{equation}{section}
\newcommand{\norm}[2]{\ensuremath{ \left\| #1 \right\|_{#2}}} 
\newcommand{\setcond}[2]{\ensuremath{\left\{#1\,\big|\, #2\right\} }}

\newcommand{\B}{\mathbb B}

\renewcommand{\S}{\mathbb S}
\newcommand{\R}{\mathbb R}
\newcommand{\N}{\mathbb N}
\newcommand{\spa}{\text{span}}

\renewcommand{\L}{\mathscr L}

\newcommand{\id}{\text{Id}}

\newcommand{\U}{\mathcal U}

\newcommand{\X}{\mathcal X}
\newcommand{\Y}{\mathcal Y}
\newcommand{\Z}{\mathcal Z}

\newcommand{\sub}{\text{sub}}

\begin{document}

\title{Linear stability estimates for Serrin's problem via a modified implicit function theorem}  

\author{Alexandra Gilsbach}

\author{Michiaki Onodera}

\affil{\small{Department of Mathematics, School of Science, Tokyo Institute of Technology}}

\date{\small\today}

\maketitle

\begin{abstract}
\noindent\textbf{Abstract.} We examine Serrin's classical overdetermined problem under a perturbation of the Neumann boundary condition. The solution of the problem for a constant Neumann boundary condition exists provided that the underlying domain is a ball. The question arises whether for a perturbation of the constant there still are domains admitting solutions to the problem. Furthermore, one may ask whether a domain that admits a solution for the perturbed problem is unique up to translation and whether it is close to the ball. 
We develop a new implicit function theorem for a pair of Banach triplets that is applicable to nonlinear problems with loss of derivatives except at the point under consideration. Combined with a detailed analysis of the linearized operator, we prove the existence and local uniqueness of a domain admitting a solution to the perturbed overdetermined problem. 
Moreover, an optimal linear stability estimate for the shape of a domain is established.
\end{abstract}



%__________________INTRO________________________%
\section{Introduction}

We study the shape of a bounded domain $\Omega\subset\mathbb{R}^n$, $n\geq 2$, in which a solution $u$ to the Dirichlet problem
\begin{subequations}
\label{eq:serrin}
\begin{align}
\begin{aligned}
 -\Delta u&=1 \quad \text{in} \ \Omega,\\
 u&=0 \quad \text{on} \ \Gamma = \partial\Omega 
\end{aligned} \label{eq:serrin-1}
\intertext{satisfies the overdetermined boundary condition}
 -\frac{\partial u}{\partial \nu}=f \quad \text{on} \ \Gamma, \label{eq:serrin-2}
\end{align}
\end{subequations}
where $\nu$ is the outer unit normal vector to $\Gamma$ and $f$ is a prescribed positive function defined on $\mathbb{R}^n$. 

The overdetermined problem \eqref{eq:serrin} arises in a shape optimization problem called the Saint-Venant problem, in which one maximizes the torsional rigidity
\begin{equation*}
 P(\Omega)=\sup_{u\in H_0^1(\Omega)\setminus\{0\}}\frac{\left(\int_\Omega u\,dx\right)^2}{\int_\Omega|\nabla u|^2\,dx}
\end{equation*}
of a bar with cross section $\Omega$, among all sets $\Omega$ of equal weighted volume
\begin{equation*}
 V(\Omega)=\int_\Omega f^2\,dx. 
\end{equation*}
The Euler-Lagrange equation, after multiplying a normalizing constant, consists in \eqref{eq:serrin}. 
In the case where $f$ is a constant, P\'olya \citep{polya1948torsional} proved that the maximizer $\Omega$ of $P$ must be a ball with the prescribed volume $V$, using the symmetric rearrangement of a function. 
This applies to a more general situation, where $f$ is radially symmetric and non-decreasing in the radial direction. 

In fact, the same symmetry result also holds for all critical points, namely, if $f$ is a constant, then 
\eqref{eq:serrin} has a solution $u$ if and only if $\Omega$ is a ball. 
In particular, for the normalized constant $f=\tfrac{1}{n}$, $\Omega$ is a ball of radius one and $u(x)=\tfrac{1}{2n}(1-|x-c|^2)$ with $c$ being the center of the ball. 
This well-known symmetry result is due to Serrin \citep{serrin1971symmetry}. 
The proof introduces the method of moving planes motivated by Alexandrov's reflection principle \citep{alexandrov1962characteristic} originally used to establish the soap bubble theorem. 
This symmetry result can be alternatively proven by an ingenious combination of the Rellich-Pohozaev integral identity and elementary inequalities (see Weinberger \citep{weinberger1971remark}, and Brandolini, Nitsch, Salani, and Trombetti \citep{brandolini2008serrin}), or by a continuous version of the Steiner symmetrization (see Brock and Henrot \citep{brock2002symmetry}). 

The objective of this paper is the stability of a domain $\Omega$ under a perturbation of the Neumann boundary condition \eqref{eq:serrin-2}, which naturally arises if one considers the torsional rigidity in anisotropic media. 
Namely, setting $\Omega_0:=\B$, the $n$-dimensional unit ball centered at the origin, and
\begin{equation}
\label{eq:g-perturbation}
 f(x)=\frac{1}{n}+g\left(\frac{x}{|x|}\right) \quad (x\in\mathbb{R}^n\setminus \{0\})
\end{equation}
with a prescribed function $g$ defined on $\Gamma_0:=\S$, where $\S = \partial\B$, we prove the existence and local uniqueness of $\Omega$ admitting a solution $u$ to \eqref{eq:serrin}, and establish a quantitative estimate of the deviation of $\Omega$ from $\Omega_0$ in terms of the perturbation $g$.

The domain deviation is measured by a function $\rho=\rho(\zeta)\in (-1,\infty)$ which defines the star-shaped bounded domain $\Omega_\rho$ enclosed by
\begin{equation}
\label{eq:boundary}
 \Gamma_\rho:=\left\{\zeta+\rho(\zeta)\zeta \mid \zeta\in\S\right\}. 
\end{equation}
A domain $\Omega$ admitting a solution to \eqref{eq:serrin} will also be referred to as a solution of the problem.

In what follows, $h^{k+\alpha}(\overline{\Omega})$ denotes the little H\"older space defined as the closure of the Schwartz space $\mathcal{S}$ of rapidly decreasing functions in $C^{k+\alpha}(\overline{\Omega})$, and similarly $h^{k+\alpha}(\Gamma)$ for a hypersurface $\Gamma$ (see Lunardi \citep{lunardi2012analytic}). 

\medskip

In order to motivate our study, let us mention several related results concerning existence and stability of solutions to \eqref{eq:serrin}. 
The existence of $\Omega$ for non-constant $f$ is known (see Bianchini, Henrot and Salani \citep{bianchini2014overdetermined}) in the case where $f$ is positively homogeneous, i.e., 
\begin{equation}
\label{gamma-homogeneity}
 f(tx)=t^\gamma f(x) \quad (t>0,\ x\in\mathbb{R}^n)
\end{equation}
for $\gamma>0$ with $\gamma\neq 1$ with $f$ being H\"older continuous on $\R^n\setminus\{0\}$. 
This condition ensures the existence of a maximizer $\Omega$ of the Saint-Venant problem, and a solution $u$ to \eqref{eq:serrin-1} then satisfies $-\partial_\nu u=\lambda f$ on $\Gamma$ with a Lagrange multiplier $\lambda>0$. 
The $\gamma$-homogeneity of $f$ allows us to control $\lambda$ by considering $t\Omega:=\{tx\mid x\in\Omega\}$, and indeed $t=\lambda^{1/(1-\gamma)}$ gives a desired domain. 
However, the $0$-homogeneous case \eqref{eq:g-perturbation} cannot be treated by this variational approach, since the dichotomy of a maximizing sequence cannot be excluded in the concentration-compactness alternative. 

Most of the existing stability results in the literature for \eqref{eq:serrin}, fitted into our context by translation and dilation, take inequalities of the form
\begin{equation}
\label{eq:general_stability}
 \|\rho\|_{L^\infty(\mathcal{S}^{n-1})}\leq C\left[\frac{\partial u_{\Omega}}{\partial \nu}+R\right]_X^\tau, 
\end{equation}
where $u_{\Omega}$ is a solution to \eqref{eq:serrin-1} in $\Omega=\Omega_\rho$ with $C^{2+\alpha}$-boundary, $0<\tau\leq 1$ and $[\,\cdot\,]_X$ denotes a norm or seminorm which measures the deviation of $-\partial_\nu u_{\Omega}$ from a constant $R>0$. 
Aftalion, Busca and Reichel \citep{aftalion1999} adopted a quantitative version of the method of moving planes and obtained a logarithmic version of \eqref{eq:general_stability} with $X=C^1(\Gamma)$. 
The method was further developed by Ciraolo, Magnanini and Vespri \citep{ciraolo2016holder}, and they obtained \eqref{eq:general_stability} for some $0<\tau<1$ in terms of the Lipschitz seminorm of $X=\text{Lip}(\Gamma)$. 
In fact, these results also hold for semilinear equations $-\Delta u=f(u)$ with $u>0$. 
On the other hand, Brandolini, Nitsch, Salani and Trombetti \citep{brandolini2008stability} made use of integral identities and proved \eqref{eq:general_stability} with $X=L^\infty(\Gamma)$ for some $0<\tau<1$. 
Moreover, they obtained an estimate of the volume of the symmetric difference of $\Omega$ and a union of balls by a weaker norm, i.e., $X=L^1(\Gamma)$. 
Note that the problem \eqref{eq:serrin} admits a domain $\Omega$ composed of a finite number of balls joined by tiny tentacles if we only control the extra boundary condition \eqref{eq:serrin-2} by the $L^1$-norm. 
Following this approach, Feldman \citep{feldman2018stability} obtained the sharp estimate
\begin{equation*}
 |\Omega\triangle\mathbb{B}|\leq C\left\|\frac{\partial u_{\Omega}}{\partial\nu}+R\right\|_{L^2(\Gamma)}, 
\end{equation*}
where $|\Omega\triangle\mathbb{B}|$ is the volume of the symmetric difference of $\Omega$ and $\mathbb{B}$ and is considered as $\|\rho\|_{L^1(\S)}$ for star-shaped $\Omega=\Omega_\rho$. 
The linear (i.e., $\tau=1$) stability estimate has also been expected in \eqref{eq:general_stability}. 
Recently, Magnanini and Poggesi proved \eqref{eq:general_stability} with $X=L^2(\Gamma)$ and $\tau=1$ for $n=2$, $\tau$ arbitrarily close to $1$ for $n=3$, and $\tau=\frac{2}{n-1}$ for $n\geq 4$ \citep{magnanini2020nearly}.

In general, for overdetermined problems, the super-subsolution method based on the maximum principle provides an existence criterion. 
In our setting, a bounded domain $\Omega$ is called a supersolution to \eqref{eq:serrin} if the unique solution $u=u_\Omega$ to \eqref{eq:serrin-1} satisfies
\begin{equation*}
 -\frac{\partial u_\Omega}{\partial\nu}\leq f \quad \text{on} \ \Gamma, 
\end{equation*}
and a subsolution is defined analogously with the opposite inequality. 
The existence of a solution, i.e., a bounded domain $\Omega$ in which $u_\Omega$ satisfies \eqref{eq:serrin-2}, is guaranteed provided there are a supersolution $\Omega_{\sup}$ and a subsolution $\Omega_{\sub}$ satisfying $\Omega_{\sub}\subset\Omega_{\sup}$. 
Typically, balls $\mathbb{B}_r$ with large or small radii $r>0$ give super- or subsolutions. 
Indeed, 
for $\Omega = \B_r$,
\begin{equation*}
 u_{\mathbb{B}_r}(x)=\frac{r^2-|x|^2}{2n}
\end{equation*}
with $-\partial_{\nu} u_{\B_r}=\frac{r}{n}$ on $\partial\B_r$
solves \eqref{eq:serrin},
and
we see that, in the $\gamma$-homogeneous setting \eqref{gamma-homogeneity} with $\gamma>1$, $\B_r$ with large (resp.\ small) $r>0$ is a supersolution (resp.\ subsolution); while for $0\leq \gamma<1$, $\B_r$ with large (resp.\ small) $r>0$ is a subsolution (resp.\ supersolution). 
Hence these balls provide an appropriate pair of super- and subsolutions only if $\gamma>1$.

\medskip 

We therefore take another approach in this paper based on an implicit function theorem, yielding
linear stability estimates with H\"older norms on both sides of the estimate, as well as the existence and local uniqueness of $\Omega$ for a given perturbation $g$ in \eqref{eq:g-perturbation}. 
We will need to exploit detailed properties of the linearized equation 
\begin{subequations}
\label{eq:linearized_problem}
\begin{align}
 \begin{aligned}
 -\Delta p&=0 && \text{in} \ \Omega_{\rho_0},\\
 \left(H-\frac{1}{f}\right)p 
 + \frac{\partial p}{\partial \nu}
&=-\varphi && \text{on} \ \Gamma_{\rho_0},
\end{aligned}
&
\label{eq:linearized_problem-1}
\\
 p =f\tilde\rho \quad\, \text{on} \ \Gamma_{\rho_0},
& 
\label{eq:linearized_problem-2}
\end{align}
\end{subequations}
where $H=H_{\Gamma_{\rho_0}}$ is the mean curvature of $\Gamma_{\rho_0}$ normalized such that $H=n-1$ for $\Omega=\mathbb{B}$. 
The linearized equation \eqref{eq:linearized_problem} is derived by substituting a solution pair
$(\Omega_{\rho_0+\varepsilon\tilde\rho}, u_{\rho_0+\varepsilon\tilde\rho})$ with formal expansions
\begin{align*}
 \Gamma_{\rho_0+\varepsilon\tilde\rho}&=\{\zeta+\left(\rho_0(\zeta)+\varepsilon\tilde\rho(\zeta)\nu(\zeta)\right)+o(\varepsilon) \mid \zeta\in\S\},\\
 u_{\rho_0+\varepsilon\tilde\rho}&=u_{\rho_0}+\varepsilon p+o(\varepsilon)
\end{align*}
into \eqref{eq:serrin} for a right hand side $f+\varepsilon\varphi$, and  equating functions of order $\varepsilon$.
Note that \eqref{eq:linearized_problem} is a decoupled system for $p$ and $\tilde\rho$, and we may consider only \eqref{eq:linearized_problem-1} for the solvability of \eqref{eq:linearized_problem}. Then \eqref{eq:linearized_problem-2} with known $p$ yields a solution $\tilde\rho$.

Recall that the implicit function theorem states that the nonlinear equation $F(\rho, g)=0$ has for each $g$ close to $g_0$ a unique solution $\rho$ near $\rho_0$ with $F(\rho_0,g_0)=0$, if
\begin{enumerate}[label=(\roman*)]
	\item the mapping $F\colon X\times Y\to Z$ is $C^1$ in a neighbourhood of $(\rho_0, g_0)$ and if
	\item the partial derivative $\partial_\rho F(\rho_0,g_0)\in \L(X,Z)$ is bijective.
\end{enumerate}
Here, $X$, $Y$ and $Z$ are Banach spaces with $X\subset Z$. 
In addition to the solution $\rho(g)$ being locally unique, the mapping $g\mapsto \rho(g)\in X$ is in $C^1$. 
In the current setting, the Neumann boundary condition \eqref{eq:serrin-2} yields such a mapping $F$, and the linearized equation $\partial_\rho F(\rho_0,g_0)[\tilde{\rho}]=\varphi$ is reflected by \eqref{eq:linearized_problem}.
However, the linearized equation \eqref{eq:linearized_problem} has a regularity defect called loss of derivatives, 
i.e.\ $\partial_\rho F(\rho_0, g_0)^{-1} \not\in \L(Z, X)$. Since solutions $\tilde{\rho}$ to \eqref{eq:linearized_problem} are less regular than $\rho_0$, and hence the typical iterative scheme in the classical implicit function theorem fails.

One method to overcome this regularity issue is the Nash-Moser theorem, a generalization of the classical implicit function theorem introduced by Nash in \citep{nash1956imbedding} and generalized by Moser in \citep{moser1961new}. The introduction of a smoothing operator combined with Newton's method for improved convergence was there shown to be a mean to overcome the regularity deficit. For the Nash-Moser theorem to work, 
\begin{enumerate}[label=(\roman*)]
	\item regularity properties are required for $F\colon X_i\times Y\to Z_i$, where $(X_i,Z_i)_i$ is a family of pairs of Banach spaces such that $X_i\subset X_{i-1}$, $Z_i\subset Z_{i-1}$. Furthermore, 
	\item a (right) inverse of $\partial_\rho F(\rho,g)$ has to exist for $(\rho,g)$ in a neighbourhood of $(\rho_0,g_0)$. 
\end{enumerate}
In this setting, for every $g$ in a neighbourhood of $g_0$, the existence of $\rho(g)$ in $X_0$ is then given. Note that there are various versions of the Nash-Moser theorem, also referred to as Nash-Moser-H\"ormander theorem. We refer as an example to the work of Baldi and Haus \citep{baldi2017nash} and the references therein.

Instead of applying the Nash-Moser theorem, we introduce a new modified version of the classical implicit function theorem, which has the constraint that a loss of derivatives may take place \emph{except} at the point 
$(\rho_0,g_0)$. We require for a pair of Banach triplets 
$X_2 \subset X_1 \subset X_0$ and $Z_2 \subset Z_1 \subset Z_0$ that
\begin{enumerate}[label=(\roman*)]
	\item for $j=1,2$, $F$ is continuous in a neighbourhood of $(\rho_0,g_0)$ from $X_{j-1}\times Y$ to $Z_{j-1}$, and that
	it is in $C^1$ in a neighbourhood of $(\rho_0,g_0)$ from $X_j \times Y$ to $Z_{j-1}$, 
	For $(\rho,g)$ in a neighbourhood of $(\rho_0,g_0)\in X_j \times Y$, we have 
	$\partial_\rho F(\rho,g)\in \L(Z_{j-1}, X_{j-1})$. Further, 
\item $F\colon X_j \times Y \to Z_j$ is Fr\'echet-differentiable at $(\rho_0,g_0)$ for $j=1,2$ and $\partial_\rho F(\rho_0,g_0)\in \L(X_j, Z_j)$ is invertible for $j=0,1,2$.
\end{enumerate}
The first point reflects the loss of regularity, the second point reflects that it does not occur at the point under consideration. 
Under these assumptions, we derive a modified implicit function theorem that yields local uniqueness of a solution $\rho(g)\in X_1$ for all $g$ in a neighbourhood of $g_0$, and
the mapping $g\mapsto \rho(g)\in X_0$ is in $C^1$.
Note that in the setting of \eqref{eq:serrin} and \eqref{eq:linearized_problem}, the loss of derivatives does indeed not occur in the case of the solution of \eqref{eq:serrin} for constant $f$, as then $\Gamma$ and $u$ are smooth.

\medskip

However, a second obstacle apart from the loss of derivatives arises. 
Due to the translational invariance of \eqref{eq:serrin}, the linearized equation \eqref{eq:linearized_problem-1} for $g=0$ and $\Omega_{\rho_0}=\B$ is not solvable for arbitrary $\varphi \in h^{2+\alpha}(\S)$, and for $\varphi = 0$ it has an $n$-dimensional space of solutions. 
This implies that the partial derivative of $F$ at $(0,0)$ is not invertible, which is necessary also for the modified implicit function theorem.
We will remove this degeneracy by imposing an additional condition 
\begin{equation}
\label{eq:barycenter}
 \int_{\Omega}x_j\,dx=0 \quad (j=1,\ldots,n), 
\end{equation}
so that the barycenter of $\Omega$ is fixed to be the origin, and by decomposing the space $h^{k+\alpha}(\S)$ into $h^{k+\alpha}(\S) = X_l\oplus K$, where
\begin{align}
\label{eq:def_subspaces_Xl}
\begin{aligned}
 X_l&:=\left\{\rho \in h^{l+\alpha}(\S) \mid \langle \rho, x_j\rangle_{L^2(\S)}=0, \quad j=1,\ldots,n\right\},\\
 K&:=\spa\left\{x_1,\ldots,x_n\right\}. 
\end{aligned}
\end{align}
This allows for a decomposition of a domain perturbation $\rho\in h^{l+\alpha}(\S)$ into $\rho_1\in X_l$, $\rho_2\in K$, as well as a decomposition of the function $g$ in \eqref{eq:g-perturbation} likewise, and 
we examine $F(\rho_1+\rho_2,g_1+g_2)=0$. 

With these preparations, we present the main result of this work.

\begin{theorem}\label{thm:main}
There exist neighbourhoods of zeros 
\[
 V\subset X_2,\, U_2 \subset h^{2+\alpha}(\S)\times K \text{ and } U_3 \subset h^{3+\alpha}(\S)\times K,
\] 
such that for all $g_1\in V$ there are unique $(\rho, g_2) = (\rho(g_1), g_2(g_1))\in U_3$ 
such that 
the following holds.
\begin{enumerate}[label=(\roman*)]
\item $\Omega_{\rho(g_1)}$ defined by \eqref{eq:boundary} admits a solution $u\in h^{3+\alpha}(\overline{\Omega_{\rho(g_1)}})$ to \eqref{eq:serrin} for \eqref{eq:g-perturbation} with $g= g_1+g_2(g_1)$, and satisfies \eqref{eq:barycenter}. 
\item $\Omega_{\rho(g_1)}$ is locally unique up to translations in the sense that there is 
$(\rho(g_1), g_2(g_1))\in U_3$ for $g_1\in V$, and if $\Omega_\rho$ with $\rho\in U_3$ admits a solution $u$ to \eqref{eq:serrin} for \eqref{eq:g-perturbation} with $g=g_2(g_1)+g_1$ and satisfies \eqref{eq:barycenter}, 
then $\rho=\rho(g_1)$. 
\item For the mapping $(\rho, g_2)\colon V \to U_3$, we have $(\rho, g_2)\in C^1(V,U_2)$ 
and the stability estimates 
\begin{equation}
\label{eq:stability_ift} 
\begin{aligned}
\norm{\rho(g_1)}{h^{2+\alpha}(\S)} & \leq C\norm{g_1}{h^{2+\alpha}(\S)}, \\
\norm{g_2(g_1)}{h^{2+\alpha}(\S)} & \leq C\norm{g_1}{h^{2+\alpha}(\S)}
\end{aligned}
\end{equation}
hold.
\end{enumerate}
\end{theorem}

While the existence of $\rho$ is guaranteed in $h^{3+\alpha}(\S)$, only the weaker norm, i.e.\ the $h^{2+\alpha}(\S)$ norm, of $\rho$ is estimated in \eqref{eq:stability_ift}.
This is due to the fact that the linear stability estimate requires $C^1$-regularity of the mapping $g_1 \mapsto \rho$, and this regularity is expected only when the image space is $h^{2+\alpha}(\S)$ due to the loss of derivatives.

\begin{remark}
The translational invariance of \eqref{eq:serrin} is mirrored in that theorem by using the decomposition into the translational part and its orthogonal complement.
In that regard, it also becomes clear why the
setting of the little H\"older spaces $h^{l+\alpha}$ instead of the H\"older spaces $C^{l+\alpha}$ is necessary:
The decomposition into subspaces is induced by the so-called spherical harmonics on $\S$. They are
dense in $h^{l+\alpha}(\S)$, but not in $C^{l+\alpha}(\S)$.
This will be further discussed in Section \ref{sect:degeneracy}.
\end{remark}

The paper is organised as follows. In Section \ref{sect:preliminaries}, we introduce the perturbed problem as well as 
derive in detail the formulation via the linearized equation \eqref{eq:linearized_problem}.
We motivate the application of an implicit function theorem to a mapping $F$ that is derived from the Neumann boundary condition. This application is obstructed by the degeneracy of the derivative of $F$ as well as the loss of derivatives. 
The degeneracy of the derivative of $F$ stemming from the inherent symmetry of \eqref{eq:serrin} will be addressed in Section \ref{sect:degeneracy}. 
There, also the decomposition for the little H\"older spaces is motivated as well as the necessity of using the setting of the little H\"older spaces.
In Section \ref{sect:modified_ift}, we will revisit the implicit function theorem and establish a modified version fitting our setting. This is then applied to the perturbed problem in Section \ref{sect:existence_stability_ift} to 
prove Theorem \ref{thm:main}.

%__________________PRELIM_______________________%
\section{Preliminaries}
\label{sect:preliminaries}

We formally set up the perturbed problem defined in \eqref{eq:serrin}. 
We want to know whether for a perturbation
$g$ there exists an open bounded domain $\Omega=\Omega(g)$ admitting a solution $u_\Omega$ to \eqref{eq:serrin} with \eqref{eq:g-perturbation}, i.e.\ $f = \frac{1}{n} + g$.

We restrict the domain $\Omega$ to be in such a way that it may be modelled as a deviation of $\B$, the domain admitting a solution to \eqref{eq:serrin} with $f=\frac{1}{n}$. 
For this reference domain $\Omega_0 = \B$, with $\partial\Omega_0=\Gamma_0= \S$,
we define the perturbed domain $\Omega_\rho$ by its $h^{m+\alpha}$-boundary $\Gamma_\rho = \partial \Omega_\rho$ in the following way. We set for $m\in\N$
\begin{equation*}
U_{\gamma,m} := \setcond{v\in h^{m+\alpha}(\S)}{\norm{v}{h^{m+\alpha}(\S)}<\gamma}, 
\end{equation*}
with $\gamma\leq 1$ sufficiently small.
Next, we define 
\begin{equation*}
\theta:\,\S\times(-1,\infty)\to \theta\left(\S\times(-1,\infty)\right), 
\quad \theta (\zeta, r) := \zeta + r\nu_0(\zeta) = \zeta + r\zeta.
\end{equation*}
In general, $\nu_\rho$ denotes the outer unit normal vector of $\Gamma_\rho$; for $\rho = 0$ we have $\nu_0(x) = x$.
Then we set 
\begin{equation*}
\Gamma_\rho = \setcond{\zeta + \rho(\zeta)\nu_0(\zeta)\in\R^n}{\zeta\in\Gamma_0} 
= \setcond{\zeta + \rho(\zeta)\zeta\in\R^n}{\zeta\in\S},
\end{equation*}
and $\rho\in U_{\gamma,m}$. $\rho$ models the velocity of the boundary, and will be used to measure how much $\Gamma_\rho$ deviates from $\Gamma_0$.
Using this, we define the diffeomorphism 
\begin{equation*}
\theta_\rho (x) := 
\begin{cases}
x + \varphi\left(|x|-1\right)\rho\left(\frac{x}{|x|}\right)\frac{x}{|x|}, \quad & \text{for } x\neq 0,\\
0 \quad & \text{for } x=0,
\end{cases}
\end{equation*}
from $\Omega_0=\B$ to $\Omega_\rho$, 
where $\varphi \colon \R\to\R$ is a smooth cut-off function with 
$0\leq \varphi(r)\leq 1$, $\varphi(r) = 1$ for $|r| \leq \frac{1}{4}$ and $\varphi(r) = 0$ for $|r| \geq \frac{3}{4}$, as well as 
$\left|\frac{d\varphi}{dr}(r)\right|\leq 4$. 
The diffeomorphism $\theta_\rho$ induces pullback and pushforward operators
\begin{equation*}
\begin{aligned}
\theta_\rho^* u & := u \circ \theta_\rho, \quad & \theta_\rho^*\colon h^{k+\alpha}(\Gamma_\rho)\to h^{k+\alpha}(\S), \\
\theta_*^\rho v & := v \circ \theta_\rho^{-1}, \quad & \theta_*^\rho\colon h^{k+\alpha}(\S)\to h^{k+\alpha}(\Gamma_\rho),
\end{aligned}
\end{equation*}
with $k\in\N\cup \{0\}$. 

Our problem now becomes the following: 
\begin{problem}
\label{problem:eins}
For $g\in h^{1+\alpha}(\S)$, is there a $\rho = \rho(g)\in U_{\gamma,2}$ such that $\Omega_\rho$ as defined above admits a solution $u_\rho$ to \eqref{eq:serrin_perturbed_rho}?
\begin{subequations}
\label{eq:serrin_perturbed_rho}
\begin{align}
\begin{aligned}
 -\Delta u_\rho			& = 1 \quad &\text{in }\Omega_\rho, \\
	  u_\rho  				& = 0 \phantom{+ g\left(\frac{x}{|x|}\right)} \quad &\text{on }\Gamma_\rho, 
\end{aligned}
&
\label{eq:serrin_perturbed_rho-1}
\\
\frac{\partial u_\rho}{\partial \nu_\rho}(x)  = \frac{1}{n} + g\left(\frac{x}{|x|}\right) \quad \text{on } \Gamma_\rho.&
\label{eq:serrin_perturbed_rho-2}
\end{align}
\end{subequations}
\end{problem}
By elliptic regularity theory, we have, for given $\rho \in U_{\gamma,2}$, the existence and uniqueness of 
a solution $u_\rho\in h^{2+\alpha}(\overline\Omega_\rho)$ when only considering \eqref{eq:serrin_perturbed_rho-1}. 
Therefore, for the examination of this problem, it is sufficient to focus on the perturbation, i.e.\ \eqref{eq:serrin_perturbed_rho-2}.

We define $F\in C(U_{\gamma,2}\times h^{1+\alpha}(\S), h^{1 + \alpha}(\S))$ by
\begin{equation}
\label{eq:function_f}
F(\rho, g) := \theta_\rho^* \left( \frac{\partial u_\rho}{\partial \nu_\rho}\right)  +\frac{1}{n} + g,
\end{equation}
where $u_\rho$ is the unique solution of \eqref{eq:serrin_perturbed_rho-1}. 
Then $\Omega_\rho$ admits a solution to \eqref{eq:serrin_perturbed_rho} for given $g\in h^{1+\alpha}(\S)$
if and only if $F(\rho, g) = 0$.

This structure tempts to use the implicit function theorem to arrive at solutions in a neighbourhood of $(0,0)$. However, we shall arrive at two obstacles. The first is the derivative $\partial_\rho F(0,0)$ not being bijective due to the inherent translational invariance, an observation that will be treated in Section \ref{sect:degeneracy}. The second is the loss of derivatives, a regularity issue of the $\rho$-derivative of $F$ that will be discussed in the following.

In view of this, note that for $g\in h^{2+\alpha}(\S)$ and for $m=2,3$, we have
\begin{equation}
\label{eq:regularity_function_f_general} 
F\in C(U_{\gamma,m}\times h^{2+\alpha}(\S), h^{m-1 + \alpha}(\S)).
\end{equation}

%-------------%
\subsection{Derivative of \texorpdfstring{$F$}{F}} 

We turn to the $\rho$-differentiability of F at a point $(\rho_0,g)$. 
Due to the loss of derivatives, we need to assume $\rho_0 \in U_{\gamma,3}$. 
We consider
\[
F(\rho_0 + \varepsilon \tilde \rho, g) - F(\rho_0, g) = A(\rho_0, g)[\varepsilon \tilde\rho] + o(\varepsilon) 
\]
for $\tilde\rho\in U_{\gamma,3} $ and $\varepsilon\to 0$.
Since $u_\rho$ in $F(\rho,g)$ lives on $\overline\Omega_\rho$, which varies for $\rho$, we consider the following approach.

Let $x\in \overline\Omega_0$.
We define the mapping $u(\rho, x) := u_{\rho}(\theta_{\rho}(x))$ 
and $y = \theta_{\rho_0}(x) \in \overline\Omega_{\rho_0}$, as well as 
$z = \theta_{\rho_0 + \varepsilon\tilde\rho}(x) \in \overline\Omega_{\rho_0+\varepsilon\tilde\rho}$.
One may show that $u(\rho,\theta_{\rho}(x))$ is differentiable with respect to $\rho$, for the procedure see e.g.\ \citep[Sect.\ 5.6]{henrot2018shape}.
Therefore, the following calculations are well-defined.

Using the Taylor expansion, we write  
\begin{align*}
u_{\rho_0+\varepsilon\tilde\rho}(z) & = u(\rho_0 + \varepsilon\tilde\rho, \theta_{\rho_0 + \varepsilon\tilde\rho}(x)) \\
& = u(\rho_0,y) + \underbrace{\partial_\rho u(\rho_0,y)\tilde\rho \left(\frac{\theta_{\rho_0}^{-1}(y)}{|\theta_{\rho_0}^{-1}(y)|}\right)}_{=: p(y)}\varepsilon \\
& \quad 
+ \partial_y u(\rho_0,y)V_{\rho_0}(y) \nu_0 \left(\frac{\theta_{\rho_0}^{-1}(y)}{|\theta_{\rho_0}^{-1}(y)|}\right)\tilde\rho \left(\frac{\theta_{\rho_0}^{-1}(y)}{|\theta_{\rho_0}^{-1}(y)|}\right)\varepsilon + o(\varepsilon)
\end{align*}
with $V_{\rho_0}(y) := \varphi\left(\left|\theta_{\rho_0}^{-1}(y)\right| -1\right)$. 
The function $p$ is the so-called shape derivative of $u_{\rho_0}$ with respect to the domain variation from $\Omega_{\rho_0}$ to $\Omega_{\rho_0+\varepsilon\tilde\rho}$. 

Next, we reformulate problem \eqref{eq:serrin_perturbed_rho} in terms of $\tilde\rho$ and $p$. 
Using the representation as before and considering the Dirichlet problem \eqref{eq:serrin_perturbed_rho-1} 
for $u_{\rho_0+\varepsilon\tilde\rho}$ as well as for $u_{\rho_0}$, and letting $\varepsilon \to 0$, we get 
\begin{align}
\label{eq:dirichlet_p}
\begin{aligned}
\Delta_y p(y)& =0 & \quad \text{in } \Omega_{\rho_0}, \\
p(y) & = - \frac{\partial u(\rho_0,y)}{\partial \nu_{\rho_0}} \left(\theta_*^{\rho_0}\tilde\rho\right)(y) \frac{1}{|\nabla N_{\rho_0}|}
& \quad \text{on } \Gamma_{\rho_0}.
\end{aligned}
\end{align}
Here, $N_\rho(x) := |x|-1 - \rho(\frac{x}{|x|})$ is defined for $x\in\theta(\S\times(-1,\infty))$
and we note that $\Gamma_\rho $ is the zero-level set of $N_\rho$. 
Therefore, for the outer unit normal vector field $\nu_{\rho_0}$ at $\Gamma_{\rho_0}$, we have 
$\nu_{\rho_0}(x) = \frac{\nabla N_{\rho_0}(x)}{|\nabla N_{\rho_0}(x)|}$ and  
$\nabla N_{\rho_0}\left(\theta_{\rho_0}(\zeta)\right). \nabla N_0(\zeta) = 1$ for $\zeta\in \S$.
We also note that  
\begin{equation*}
\theta_*^{\rho_0}(\nu_0) = (\nu_{\rho_0}.\nu_0)\nu_{\rho_0} + \tau_{\rho_0} 
= \frac{1}{|\nabla N_{\rho_0}|}\nu_{\rho_0} + \tau_{\rho_0},
\end{equation*}
where $\tau_{\rho_0}$ is a tangent vector field. Here, we used 
$(\nu_{\rho_0}.\nu_0) = \frac{1}{\left|\nabla N_{\rho_0}\left(\theta_{\rho_0}(\zeta)\right)\right|}$. 

\begin{remark}
The regularity of $u_{\rho_0}$ and $\nu_{\rho_0}$, $\rho_0\in U_{\gamma,3}$, imposes restrictions on the regularity of $p$ and in general, we can only expect $p\in h^{2+\alpha}(\Gamma_{\rho_0})$. 
We have $\partial_{\nu_{\rho_0}}{ u_{\rho_0}} \in h^{2+\alpha}(\Gamma_{\rho_0})$ 
and for the mean curvature $H_{\rho_0}\in h^{1+\alpha}(\Gamma_{\rho_0})$, but in general, no more. If, however, we are in the setting for $\rho_0=0$, then we have $\partial_{\nu_{0}}{u_{0}} \in C^{\infty}(\S)$, 
$-\partial_{\nu_{0}}{u_{0}} = \frac{1}{n}$ and $p = \frac{1}{n}\tilde\rho$, 
thus $p\in h^{3+\alpha}(\S)$, the same regularity as $\tilde\rho$. 
\end{remark}

Now we calculate the Fr\'echet-derivative of $F$. With the notation as before, let 
$x\in \S$, $y = \theta_{\rho_0}(x) \in \Gamma_{\rho_0}$ and 
$z = \theta_{\rho_0 + \varepsilon\tilde\rho}(x) \in \Gamma_{\rho_0+\rho}$. Note that in this case, $V_{\rho_0}(y)=1$ and $\frac{\theta_{\rho_0}^{-1}(y)}{|\theta_{\rho_0}^{-1}(y)|}=x$. 

First note that there exists a tangent vector $\tau$ at $y \in \Gamma_{\rho_0}$ such that 
\[
\nu_{\rho_0 + \varepsilon\tilde\rho} (z) = \nu_{\rho_0}(y) + \varepsilon\tau(y) + o(\varepsilon).
\]
Next, we calculate 
\begin{align*}
\partial_{z_i} u_{\rho_0 + \varepsilon\tilde\rho}(z) & = \partial_{z_i} u({\rho_0 + \varepsilon\tilde\rho}, \theta_{\rho_0 + \varepsilon\tilde\rho}(x)) \\
& = \frac{\partial}{\partial y_i}u(\rho_0,y) + \varepsilon \frac{\partial}{\partial y_i} p(y)
+ \varepsilon \frac{\partial}{\partial y_i}\frac{\partial}{\partial y_k} \left(u(\rho_0,y)\right) \theta_*^{\rho_0}\left(\nu_0\tilde\rho\right)(y) 
+ o(\varepsilon) .
\end{align*}
This implies 
\begin{align*}
& \quad \nabla_z u_{\rho_0 + \varepsilon\tilde\rho}(z). \nu_{\rho_0 + \varepsilon\tilde\rho} (z) 
 = \partial_{z_i} u_{\rho_0 + \varepsilon\tilde\rho}(z)\nu_{\rho_0 + \varepsilon\tilde\rho}^i (z) \\
& = \left[ \frac{\partial}{\partial y_i}u(\rho_0,y) + \varepsilon \frac{\partial}{\partial y_i} p(y)
+ \varepsilon \frac{\partial}{\partial y_i}\frac{\partial}{\partial y_k} \left(u(\rho_0,y)\right) \theta_*^{\rho_0}\left(\nu_0\tilde\rho\right)(y)  \right] 
\left[ \nu_{\rho_0}(y) + \varepsilon\tau(y)\right]^i
+ o(\varepsilon) 
\\ 
& = \frac{\partial}{\partial y_i}u(\rho_0,y)\nu_{\rho_0}^i(y) 
+ \varepsilon\frac{\partial}{\partial y_i}u(\rho_0,y)\tau_{\rho_0}^i(y)
 + \varepsilon \frac{\partial}{\partial y_i} p(y)\nu_{\rho_0}^i(y)  \\
& \quad + \varepsilon \frac{\partial}{\partial y_i}\frac{\partial}{\partial y_k} \left(u(\rho_0,y)\right) \theta_*^{\rho_0}\left(\tilde\rho\right)(y)\left[ \frac{1}{|\nabla N_{\rho_0}|}\nu_{\rho_0} + \tau_{\rho_0}\right]\nu_{\rho_0}^i(y) 
+ o(\varepsilon) 
\\ 
& = \frac{\partial}{\partial \nu_{\rho_0}}u(\rho_0,y) + \varepsilon \frac{\partial}{\partial \nu_{\rho_0}} p(y) 
+ \varepsilon \frac{\partial^2}{\partial \nu_{\rho_0}^2}u(\rho_0,y)\theta_*^{\rho_0}\left(\tilde\rho\right)(y)\frac{1}{|\nabla N_{\rho_0}|} \\
& \quad
+ \varepsilon \frac{\partial}{\partial \tau_{\rho_0}}\frac{\partial}{\partial \nu_{\rho_0}}u(\rho_0,y)\theta_*^{\rho_0}\left(\tilde\rho\right)(y)
+ o(\varepsilon) 
\\
& = \frac{\partial}{\partial \nu_{\rho_0}}u(\rho_0,y) + \varepsilon \frac{\partial}{\partial \nu_{\rho_0}} p(y) 
+ \varepsilon \left( -1 - H_{\rho_0}\frac{\partial}{\partial \nu_{\rho_0}}u(\rho_0,y)\right)
\theta_*^{\rho_0}\left(\tilde\rho\right)(y)\frac{1}{|\nabla N_{\rho_0}|} \\
& \quad
+ \varepsilon \frac{\partial}{\partial \tau_{\rho_0}}\frac{\partial}{\partial \nu_{\rho_0}}u(\rho_0,y)\theta_*^{\rho_0}\left(\tilde\rho\right)(y)
+ o(\varepsilon), 
\end{align*}
where in the last step we used the identity
$
\Delta u_{\rho_0} = \Delta_{\Gamma_{\rho_0}} u_{\rho_0} + \partial^2_{\nu_{\rho_0}} u_{\rho_0} 
+ H_{\rho_0}\partial_{\nu_{\rho_0}} u_{\rho_0},
$
with $\Delta_{\Gamma_{\rho_0}}$ being the Laplace-Beltrami operator and $H_{\rho_0}$ the mean curvature on $\Gamma_{\rho_0}$.
Using the Dirichlet boundary condition for $p$ in \eqref{eq:dirichlet_p}, we arrive at 
\begin{align*}
F(\rho_0 + \varepsilon\tilde\rho, g) & = F(\rho_0, g) 
+ \varepsilon \theta_{\rho_0}^*\left[
\frac{\partial}{\partial \nu_{\rho_0}} p + H_{\rho_0} p - \frac{\theta_*^{\rho_0}\tilde\rho}{|\nabla N_{\rho_0}|} + \frac{\partial}{\partial \tau_{\rho_0}}\frac{\partial}{\partial \nu_{\rho_0}}u(\rho_0,y)\theta_*^{\rho_0}\tilde\rho
\right] + o(\varepsilon).
\end{align*}
We see that the term 
\[
\theta_{\rho_0}^*\left[
\frac{\partial}{\partial \nu_{\rho_0}} p + H_{\rho_0} p - \frac{\theta_*^{\rho_0}\tilde\rho}{|\nabla N_{\rho_0}|} + \frac{\partial}{\partial \tau_{\rho_0}}\frac{\partial}{\partial \nu_{\rho_0}}u(\rho_0,y)\theta_*^{\rho_0}\tilde\rho
\right]
\]
is well-defined and lies in $h^{1+\alpha}(\S)$ even when only assuming $\tilde\rho\in h^{2+\alpha}(\S)$. Note that to verify this, one also needs to take into account the impact of the regularity assumption of $\tilde\rho$ on the regularity of the solution $p$ of \eqref{eq:dirichlet_p} as mentioned in Remark \ref{rem:regularity_f}. 
This gives us the following lemma. 

\begin{lemma}
The mapping $F$ as defined in \eqref{eq:function_f} satisfies 
	\begin{equation}
	\label{rem:regularity_f}
	F \in C(U_{\gamma,2}\times h^{1+\alpha}(\S), h^{1 + \alpha}(\S)) \cap C^1(U_{\gamma,3}\times h^{1+\alpha}(\S), h^{1 + \alpha}(\S)).
	\end{equation}
Furthermore, we have the following.
\begin{enumerate}
	\item 
		The $\rho$-Fr\'echet-derivative of $F$ at $(\rho_0, g)\in U_{\gamma,3}\times h^{1+\alpha}(\S)$ is 
	\begin{align*}
	\begin{aligned}
	& \partial_\rho F(\rho_0, g) \left[\tilde\rho\right] \\
	&  = \theta_\rho^* \left[
				\frac{\partial p}{\partial\nu_{\rho_0}} + H_{\rho_0} p - \frac{\theta_*^{\rho_0}\tilde\rho}{|\nabla N_{\rho_0}|}
			+ \frac{\partial}{\partial\tau_{\rho_0}} \left(\frac{\partial u_{\rho_0}}{\partial\nu_{\rho_0}}\right)
						\theta_*^{\rho_0}\tilde\rho
						\right] \in h^{1+\alpha}(\S),
	\end{aligned}
	\end{align*}
	where $p$ is a unique solution to \eqref{eq:dirichlet_p}, i.e.
	\begin{align*}
	\begin{aligned}
	\Delta p& =0 & \quad \text{in } \Omega_{\rho_0}, \\
	p & = - \frac{\partial u_{\rho_0}}{\partial \nu_{\rho_0}} \theta_*^{\rho_0}\tilde\rho \frac{1}{|\nabla N_{\rho_0}|}
	& \quad \text{on } \Gamma_{\rho_0}.
	\end{aligned}
	\end{align*}
	\item We have
	\begin{align*}
	\partial_\rho F(0,0)		& \in \mathcal L(h^{3+\alpha}(\S), h^{2+\alpha}(\S)) \text{ and }\\
	\partial_\rho F(\rho, g)& \in \mathcal L(h^{3+\alpha}(\S), h^{1+\alpha}(\S)).
	\end{align*}
	\item
	The linear operator $\partial_\rho F(\rho, g)$ has the extension 
	\begin{equation*}
	\partial_\rho F(\rho, g)\in \mathcal L(h^{2+\alpha}(\S), h^{1+\alpha}(\S)).
	\end{equation*}
\end{enumerate}
\end{lemma}
In view of \eqref{eq:regularity_function_f_general}, one may verify that for $m=2,3$ and $g\in h^{2+\alpha}(\S)$, 
one has
\begin{equation*}
F\in C(U_{\gamma,m}\times h^{2+\alpha}(\S), h^{m-1 + \alpha}(\S)) \cap 
C^1(U_{\gamma,m+1}\times h^{2+\alpha}(\S), h^{m-1 + \alpha}(\S)). 
\end{equation*}

\begin{remark}
We have the following characterisation of bijectivity of the $\rho$-Fr\'echet-deriva-tive of $F$ at a point $(\rho, g)$ with $F(\rho, g) = 0$, i.e.\ when $\Omega_\rho$ is a solution to \eqref{eq:serrin_perturbed_rho} for 
$g\in h^{1+\alpha}(\S)$:

The extended operator $\partial_\rho F(\rho, g)$, with
\[
\partial_\rho F \in C\left(U_{\gamma,3}\times{h^{1+\alpha}(\S)}, \mathcal L\left(h^{2+\alpha}(\S), h^{1+\alpha}(\S)\right)\right),
\]
has the bounded inverse 
\begin{equation*}
\partial_\rho F(\rho, g)^{-1} \in \mathcal L\left(h^{1+\alpha}(\S), h^{2+\alpha}(\S)\right)
\end{equation*}
if and only if the boundary problem 
\begin{align}
\label{eq:system_p_varphi}
\begin{aligned}
-\Delta p & = 0 &\text{ in }\Omega_\rho \\
\left( H_\rho - \frac{1}{\frac{1}{n} + \theta_*^\rho g}\right) p + \frac{\partial p}{\partial \nu_\rho} 
		& = - \varphi & \text{ on }\Gamma_\rho.
\end{aligned}
\end{align} 
is uniquely solvable for any $\varphi\in h^{1+\alpha}(\Gamma_\rho)$.
Unique solvability of \eqref{eq:system_p_varphi} is given provided that 
$ \left( H_\rho - \frac{1}{\frac{1}{n} + \theta_*^\rho g }\right) > 0,$
see \citep[Thm.\ 6.31]{gilbarg2001elliptic}, 
in which case we would have 
$\norm{p}{h^{2+\alpha}(\overline\Omega_\rho)} \leq C \norm{\varphi}{h^{1+\alpha}(\Gamma_\rho)}$. 
This does not hold in the current setting. Thus, we have to examine the bijectivity in a different manner. 
\end{remark}

%__________________DEGEN________________________%
\section{Degeneracy of \texorpdfstring{$\partial_\rho F$}{derivative of F}}
\label{sect:degeneracy}

\subsection{Non-Bijectivity of the partial derivative of \texorpdfstring{$F$}{F}}

We examine the $\rho$-derivative of $F$ at $(\rho, g) = (0,0)$, as we merely require the existence of an inverse of $\partial_\rho F(0,0)$ to use the modified implicit function theorem, Theorem \ref{thm:ift_part2}. We have 
\begin{equation*}
\partial_\rho F(0,0)[\tilde\rho] 
 = -\frac{1}{n}\tilde\rho + \frac{1}{n}\mathscr N\tilde\rho, \quad \text{for }\tilde\rho\in h^{3+\alpha}(\S).
\end{equation*}
Here, $\mathscr N$ denotes the Dirichlet-to-Neumann operator on the sphere $\S$.

\begin{definition}
Let $\varphi\in h^{l+\alpha}(\S)$, $l\geq 2$ arbitrary. The Dirichlet-to-Neumann operator on the sphere 
$\mathscr N\colon h^{l+\alpha}(\S) \to h^{l-1+\alpha}(\S)$ is defined as
\[
\mathscr N \varphi := \partial_\nu u, \text{ with } u \text{ the unique solution of } 
\begin{cases}
-\Delta u & =  0 \quad \text{in } \B, \\
\hfil u & =  \varphi \quad \text{on }\S.
\end{cases}
\]
\end{definition}
We see that 
\begin{equation*}
\mathscr N \in \L\left(h^{l+\alpha}(\S), h^{l-1+\alpha}(\S) \right),\,l\in\N.
\end{equation*}

It suffices to test bijectivity of $\partial_\rho F(0,0)$ for $\tilde\rho\in H_k$, $k\in \N\cup\{0\}$, 
where 
\[
H_k = \spa\setcond{h_{k,j}}{ j=1,\ldots,d_k^n}, \quad d_k^n<\infty,
\]
is the set of harmonic homogeneous polynomials on the unit sphere of degree $k$. 
Indeed, the $h_{k,j}$, $k\in \N\cup\{0\},\, j=1,\ldots,d_k^n$, 
form an orthonormal basis of $L^2(\S)$, see e.g. \citep[Thm.\ 2.53]{folland1995introduction}. 
One may show that 
\[
\mathscr H:=\spa \setcond{ h_{k,j}}{k\in \N\cup\{0\},\, j=1,\ldots,d_k^n}
\] 
is dense in $h^{l+\alpha}(\S)$, which is why it is sufficient to consider $\tilde\rho\in \mathscr H$.

Therefore, let $\tilde\rho = h_{k,j}$ be a harmonic homogeneous polynomial on the unit sphere of order 
$k\in \N\cup\{0\}$, with $j\in\{1,\ldots,d_k^n\}$. 
We get
\begin{equation*}
\partial_\rho F(0,0)[h_{k,j}] = - \frac{1}{n}h_{k,j} + \frac{k}{n}h_{k,j} = \frac{k-1}{n}h_{k,j} 
\begin{cases}
 = 0 & \text{if }k=1, \\
\neq 0 & \text{else}.
\end{cases}
\end{equation*}
This shows that $\partial_\rho F(0,0)$ is not bijective and that its kernel is 
\begin{equation}
\label{eq:derivative_f_kernel}
\text{ker}\left(\partial_\rho F(0,0)\right) = 
\spa \setcond{h_{1,j}}{j= 1,\ldots, n},
\end{equation}
with $\dim(\text{ker}(\partial_\rho F(0,0))) = n < \infty$. 
Furthermore, we see that the range of $\partial_\rho F(0,0)$ is 
\[
 \text{range}\left(\partial_\rho F(0,0)\right) = 
\overline{\spa \setcond{h_{k,j}}{k\in\N_{\geq 2}\cup \{0\}, j= 1,\ldots, d_k^n}}^{\norm{\cdot}{h^{2+\alpha}(\S)}}.
\]

\begin{notation}
In view of the calculations to come, we define 
\begin{equation*}
X_l = \overline{\spa \setcond{h_{k,j}}{k\in\N_{\geq 2}\cup \{0\}, j= 1,\ldots, d_k^n}}^{\norm{\cdot}{h^{l+\alpha}(\S)}}, \text{ for }l\in\N.
\end{equation*}
Note that this is equivalent to defining 
\[
X_l := \setcond{\rho\in h^{l+\alpha}(\S)}{\langle{\rho},h_{1,j}\rangle_{L^2(\S)}=0,\,j=1,\ldots,n},
\]
as in \eqref{eq:def_subspaces_Xl}. 
To confirm this, also note that $ h_{1,j}(x) = \omega_n^{-\frac{1}{2}} x_j$, $j=1,\ldots,n$.
Since $X_l$ is a finite-codimensional subspace of $h^{l+\alpha}(\S)$, we have $X_l \oplus X_l^\bot = h^{l+\alpha}(\S)$, 
where $X_l^\bot$ denotes the orthogonal complement of $X_l$. Because $\dim(X_l^\bot)=n<\infty$ for all $l\in\N$, we do not need to differentiate between those spaces depending on $l$, 
as we have $X_l^\bot\cong \R^n$. Therefore, we may define 
\[
K:= \spa \setcond{h_{1,j}}{j= 1,\ldots, n}
\]
and get $X_l \oplus K = h^{l+\alpha}(\S)$.
Finally, we denote 
\[
\U_{l} := \setcond{\rho\in U_{\gamma,l}}{\langle{\rho},{h_{1,j}}\rangle_{L^2(\S)}=0,\,j=1,\ldots,n},
\]
as well as $\U_l^\bot$ for the subset of $K$ such that $U_{\gamma,l} = \U_l \oplus \U_l^\bot$.
\end{notation}

\begin{remark}
That the kernel of $\partial_\rho F(0,0)$ is non-trivial is not surprising, in fact it is an obvious property resulting from the translational invariance of \eqref{eq:serrin_perturbed_rho} with $g=0$. 
The overdetermined problem is in this case solvable for any translated sphere $\B + c := B_1(c)$, with $c\in\R^n$, and with solution $ u_c(x) = \frac{1}{2n}\left(1 - |x-c|^2\right)$. 

To show the connection to the kernel of $\partial_\rho F(0,0)$, we find $\rho \in U_{\gamma,3}$ such that $\Gamma_\rho = \partial\left(\B+c\right)$. With the ansatz 
$
x + \rho(x)x = y + c,\ x,y\in\S, 
$
we arrive at 
$\rho(x) = \rho(c,x) = x.c - 1 \pm \sqrt{1 - |c|^2 + (x.c)^2}$ for any $|c|\leq 1$. For $c=te_j$,
we then arrive at 
\begin{equation*}
\frac{d}{dt}\rho(te_j, x)\big|_{t=0}
= \left[x_j + \frac{1}{2\sqrt{1 - t^2 + t^2x_j^2}} \left(-2t + 2tx_j^2\right)\right]_{t=0} 
= x_j.
\end{equation*}
Now, because of the translational invariance of \eqref{eq:serrin_perturbed_rho}, we have $F(\rho(c,x),0) = F(0,0) = 0$,
which implies  
\[
\partial_\rho F(0,0)\left[ \frac{d}{dt}\left( \rho(te_j,\cdot)\right)\big|_{t=0} \right] = 0
\]
and we arrive at $\partial\rho F(0,0)[x_j] = 0$ for $j\in\{1,\ldots,n\}$. This coincides with \eqref{eq:derivative_f_kernel}.
\end{remark}

\subsection{Re-formulation to eliminate degeneracy}

To eliminate the problem of non-bijectivity, i.e.\ the degeneracy of the problem \eqref{eq:system_p_varphi}, we need to eliminate the translation invariance in the original problem \eqref{eq:serrin_perturbed_rho}.
Therefore, we replace $F$ as defined in \eqref{eq:function_f} by
a mapping 
\[
G \in C(\U_{2} \times \U_{2}^\bot \times X_1\times K, X_1\times \R^n\times K),
\]

\begin{align*}
\begin{aligned}
G & (\rho_1,\rho_2,g_1,g_2) := \begin{pmatrix}
	G_0 \\
	G_1 \\
	\vdots \\
	G_n \\
	G_{n+1}
\end{pmatrix} (\rho_1,\rho_2,g_1,g_2)
: = \begin{pmatrix}
	P_2F(\rho_1+\rho_2,g_1+g_2) \\
	\int_{\Omega_{\rho_1+\rho_2}}  x_1 \,\mathrm dx
	\\
	\vdots
	\\
	\int_{\Omega_{\rho_1+\rho_2}}  x_n \,\mathrm dx \\
	(\id - P_2)F(\rho_1+\rho_2,g_1+g_2)
\end{pmatrix}.
\end{aligned}
\end{align*} 
$P_l\colon h^{l+\alpha}(\S)\to X_l$ denotes the projection onto $X_l$. If it is clear which $l\in\N$ is to be used, we write $P$ for $P_l$. 

Note that for $m=2,3$, and $g_1\in X_2$, $G$ is also well-defined and we have
\begin{align*}
G \in C\left(\U_{m} \times \U_{m}^\bot \times X_2\times K, X_{m-1}\times \R^n\times K\right).
\end{align*}
By the condition $\int_{\Omega_\rho}  x_i \,\mathrm dx=0$ for $i=1,\ldots,n$, we achieve that the center of mass of 
$\Omega_\rho$ is in the origin and thus eliminate the possibility of translations, and thus, admissible $\rho$ will be in the set 
\[
\mathcal M = \setcond{\rho \in U_{\gamma,3}}{\int_{\Omega_\rho}  x_i \,\mathrm dx=0,\, i=1,\ldots,n}.
\]
As a direct consequence, we have

\begin{lemma}
$\Omega_\rho$ with barycenter zero admits a solution to \eqref{eq:serrin_perturbed_rho} for given 
$g\in h^{1 + \alpha}(\S)$ if and only if $G(\rho_1,\rho_2,g_1, g_2) = 0$, with $\rho = \rho_1+\rho_2$ and $g = g_1+g_2$. 
\end{lemma}

\subsection{Bijectivity of the partial derivative of \texorpdfstring{$G$}{F}}

The mapping $G$ has the following regularity properties.
\begin{lemma}
\label{lem:regularity_g}
$G$ is Fr\'echet-differentiable as a map from $\U_{3} \times \U_{3}^\bot\times K$ 
to $X_1\times \R^n\times K$.  
We have 
\[
\partial_{\rho_1,\rho_2,g_2}G(\rho_1,\rho_2,g_1, g_2)\in 
\L(X_3\times K \times K, X_1\times \R^n\times K),
\]
for $(\rho_1+\rho_2,g_1+g_2) \in U_{\gamma,3} \times h^{1 + \alpha}(\S)$, 
which can be extended to 
\[
\partial_{\rho_1,\rho_2,g_2}G(\rho_1,\rho_2,g_1, g_2)\in 
\L(X_2\times K \times K, X_1\times \R^n\times K).
\]
Furthermore,  
\begin{align*}
G\in\  
& C(\U_{2} \times \U_{2}^\bot \times X_1\times K, X_{1}\times \R^n\times K) \\
& \cap\  
 C^{1} (\U_{3} \times \U_{3}^\bot \times X_1\times K, X_1\times \R^n\times K) .
\end{align*}
\end{lemma}
In view of the application of the modified function theorem introduced in Section \ref{sect:modified_ift}, we
also need the following observation concerning the regularity of $G$.
For $m=2,3$, we have
\begin{equation}
\label{eq:regularity_g_general}
\begin{aligned}
G\in\  
& C\left(\U_{m} \times \U_{m}^\bot \times X_2\times K, X_{m-1}\times \R^n\times K\right) \\
& \cap\  
 C^{1} \left(\U_{m+1} \times \U_{m+1}^\bot \times X_2\times K, X_{m-1}\times \R^n\times K\right).
\end{aligned}
\end{equation}
as well as for the extension of the partial derivative 
\begin{equation}
\label{eq:regularity_g_derivative_extension}
\partial_{\rho_1,\rho_2,g_2}G(\rho_1,\rho_2,g_1, g_2)\in 
\L(X_m\times {K\times K}, X_{m-1}\times \R^n\times K), 
\end{equation}
where $(\rho_1,\rho_2,g_1, g_2) \in X_{m+1}\times K \times X_2\times K$.

\begin{proof}
We have for $i=1,2$ and 
$\tilde\rho_1 \in X_3\subset h^{3+\alpha}(\S)$, $\tilde\rho_2 \in K \subset h^{3+\alpha}(\S)$ 
\begin{align*}
 \partial_{\rho_i} G_0 (\rho_1,\rho_2,g_1, g_2)[\tilde\rho_i] 
& = 	\partial_{\rho_i} PF(\rho_1+\rho_2,g_1+g_2)[\tilde\rho_i] & \in X_1,  \\
 \partial_{\rho_i} G_{n+1} (\rho_1,\rho_2,g_1, g_2)[\tilde\rho_i] 
& = 	\partial_{\rho_i} (\id-P)F(\rho_1+\rho_2,g_1+g_2)[\tilde\rho_i] & \in K 
\end{align*}
and 
further for $j = 1,\ldots,n$
\begin{equation*}
\partial_{\rho_i} G_j (\rho_1,\rho_2,g_1, g_2)[\tilde\rho_i] 
= \int_{\Gamma_{\rho_1+\rho_2}}  \frac{1}{|\nabla N_\rho|}\theta_*^{\rho} \tilde\rho_i \cdot \sigma_j\,\mathrm d\sigma, \quad \text{for } j=1,\ldots,n.
\end{equation*}
These expressions are still well-defined and of the same regularity for $\tilde\rho\in h^{2+\alpha}(\S)$, 
implying 
the existence of an extension of $\partial_{\rho_1,\rho_2} G(\rho_1,\rho_2,g_1, g_2)$ onto $X_2\times K$ 
and thus
\[ 
\partial_{\rho_1,\rho_2} G (\rho_1,\rho_2,g_1, g_2)\in \L(X_3\times K, X_1\times\R^n\times K)\cap \L(X_2\times K, X_1\times\R^n\times K).
\]
For the $g_2$-partial derivative and $\tilde {g_2} \in K$, 
we get 
\begin{align*}
& \partial_{g_2} G_0 (\rho_1,\rho_2,g_1, g_2)[\tilde {g_2}] 
= 	\partial_{g_2} PF (\rho_1+\rho_2,g_1+g_2)[\tilde {g_2}], \\
& \partial_{g_2} G_j (\rho_1,\rho_2,g_1, g_2)[\tilde {g_2}] = 0 \quad\text{and} \\
& \partial_{g_2} G_{n+1} (\rho_1,\rho_2,g_1, g_2)[\tilde {g_2}] 
= 	\partial_{g_2} (\id-P)F (\rho_1+\rho_2,g_1+g_2)[\tilde {g_2}].
\end{align*}
This implies 
$\partial_{g_2} G (\rho_1,\rho_2,g_1, g_2)\in \L(K, X_1\times\R^n\times K)$.
\end{proof}

At zero, the partial derivative $\partial_{\rho_1,\rho_2, g_2} G$ is bijective. We abbreviate $(0,0,0,0)$ by $(0)$.
\begin{lemma}
\label{lem:regularity_g_zero}
$G$ is Fr\'echet-differentiable at $(\rho_1,\rho_2,g_1, g_2) = (0)$, and we have 
\begin{align*}
&\partial_{\rho_1,\rho_2,g_2} G(0) \in \L(X_3\times K \times K, X_2\times\R^n\times K)
\end{align*}
with 
\begin{align}
\label{eq:derivative_g}
\begin{aligned}
\partial_{\rho_1,\rho_2, g_2} G(0)[\tilde\rho_1, \tilde \rho_2, \tilde g_2] 
& = \begin{pmatrix}
		\frac{1}{n}\left(\mathscr N - \id \right)\tilde\rho_1 \\
		\omega_n^{1/2}\tilde\rho_2 \\
		\tilde g_2
\end{pmatrix}.
\end{aligned}
\end{align}
Indeed, for arbitrary $m\in\N$, we have 
\begin{equation*}
\partial_{\rho_1,\rho_2,g_2} G(0) \in \L(X_{m+1}\times K \times K, X_m\times\R^n\times K) .
\end{equation*}
Further, $\partial_{\rho_1,\rho_2, g_2} G(0)$ is invertible with 
\[
\partial_{\rho_1,\rho_2, g_2} G(0)^{-1} \in \L(X_m\times\R^n\times K, X_{m+1}\times K \times K) 
\]
for $m,\in\N$, and 
\begin{equation}
\label{eq:derivative_g_inverse}
\partial_{\rho_1,\rho_2, g_2} G(0)^{-1}[\phi, \alpha_1,\ldots,\alpha_n, \psi] 
= \begin{pmatrix}
n\left(\left(\mathscr N -1\right)\big|_{X_2} \right)^{-1}\phi \\
\omega_n^{-1/2}\sum_{j=1}^n \alpha_j h_{1,j} \\
\psi
\end{pmatrix} \in X_{m+1}\times K \times K 
\end{equation}
where $\phi \in X_m$, $\alpha_j\in \R$ for $j=1,\ldots,n$ and $\psi\in K$.    
\end{lemma} 

\begin{proof}
Let $m\in \N$.
Let $(\tilde\rho_1, \tilde \rho_2, \tilde g_2)\in  X_{m+1} \times  K \times K$ and $j=1,\ldots,n$. 
We calculate
\begin{align*}
& \partial_{\rho_1} G_0 (0)[\tilde\rho_1] 
= 	\partial_{\rho_1} PF(0)[\tilde\rho_1] = \frac{1}{n}\left(\mathscr N - \id\right)\tilde\rho_1, \\
& \partial_{\rho_2} G_0 (0)[\tilde\rho_2] 
= 	\partial_{\rho_2} PF(0)[\tilde\rho_2] = 0, \\
& \partial_{\rho_1} G_j (0)[\tilde\rho_1] 
= \int_{\S}  \tilde\rho_1\sigma_j\,\mathrm d\sigma 
= \omega_n^{1/2}\int_{\S}  \tilde\rho_1 h_{1,j}\,\mathrm d\sigma = 0, \\
& \partial_{\rho_2} G_j (0)[\tilde\rho_2] 
= \int_{\S}  \tilde\rho_2\sigma_j\,\mathrm d\sigma 
= \omega_n^{1/2}\int_{\S}  \tilde\rho_2 h_{1,j}\,\mathrm d\sigma = \omega_n^{1/2}\left(\tilde\rho_2\right)_j, \\
& \partial_{\rho_i} G_{n+1} (0)[\tilde\rho_i] 
= 	\partial_{\rho_i} (\id-P)F(0)[\tilde\rho_i] = 0, \quad \text{for } i=1,2.
\end{align*}
$\mathscr N$ again denotes the Dirichlet-to-Neumann operator. Recall that the linear operator $\left(\mathscr N - \id\right)$ is bijective as an operator in $\L(X_{m+1}, X_m)$. 
For the $g_2$-partial derivative, we find 
\begin{align*}
& \partial_{g_2} G_0 (0)[\tilde {g_2}] 
= 	\partial_{g_2} PF (0)[\tilde {g_2}] = 0, \\
& \partial_{g_2} G_j (0)[\tilde {g_2}] = 0 \quad\text{and} \\
& \partial_{g_2} G_{n+1} (0)[\tilde {g_2}] 
= 	\partial_{g_2} (\id-P)F (0)[\tilde {g_2}] = \tilde g_2.
\end{align*}
This implies \eqref{eq:derivative_g}.
We directly arrive at \eqref{eq:derivative_g_inverse} and also at the regularity properties of $\partial_{\rho_1,\rho_2,g_2}G(0)^{-1}$. 
\end{proof}

%__________________modIFT________________________%
\section{A modified implicit function theorem}
\label{sect:modified_ift}

To arrive at the existence, uniqueness and at a stability result for Problem \ref{problem:eins}, we introduce a modified version of the implicit function theorem, Theorem \ref{thm:ift_part2}. Because of the regularity issues stated in Remark \ref{rem:regularity_f}, we are not able to apply the classical implicit function theorem. 
In preparation, we need

\begin{theorem}\label{thm:ift_part1}
Assume the following. 
\begin{enumerate}[label=(\Roman*)]
	\item Let $\X_0, \X_1, \Y, \Z_0, \Z_1$ be Banach spaces with $\X_1 \hookrightarrow \X_0$ and 
	$\Z_1 \hookrightarrow \Z_0$. 
	Let $D_1\subset D_0$ be open sets such that $(0,0)\in D_j \subset \X_j\times \Y$ for $j=0,1$.
	\item {Let $F\in C^1(D_1, \Z_0)\cap C(D_0, \Z_0)$ with $F(0,0)=0$ and 
	$\partial_x F\in C(D_1, \L(\X_0, \Z_0))$, which is to be understood such that for $(x,y)\in D_1$, 
	the partial deriative $\partial_x F(x,y)\in \L(\X_1, \Z_0)$ can be extended to 
	$\overline{\partial_x F}(x,y)\in \L(\X_0, \Z_0)$ and $\overline{\partial_x F} \in C(D_1, \L(\X_0, \Z_0))$.}
	\item We have $F\colon D_1\to \Z_1$ and $F$ is Fr\'echet-differentiable at $(0,0)$, hence 
	$\partial_x F(0,0)\in \L(\X_1, \Z_1)$ and $\partial_y F(0,0)\in \L(\Y, \Z_1)$.
	\item The inverse $\partial_x F(0,0)^{-1} \in \L(\Z_1, \X_1)\cap \L(\Z_0, \X_0)$ exists.
\end{enumerate}
Then there exist neighbourhoods of zero $0\in U_0\subset \X_0$, $0\in U_1\subset \X_1$ and $0\in V\subset \Y$, as well as a function $u\colon V\to U_0$ such that 
\begin{enumerate}[label=(\roman*)]
	\item $F(u(y),y)=0$ for all $y\in V$, $u(0)=0$, and 
	\item for $x_1, x_2 \in U_1$, $y\in V$ such that $F(x_i, y)=0$ for $i=1,2$, we have $x_1 = x_2$.
\end{enumerate}
\end{theorem}

\begin{proof}
Let $\varepsilon, \delta >0$ -- we will redefine both later -- and define
\begin{align*}
U_1 := \setcond{x\in \X_1}{\norm{x}{\X_1}\leq \varepsilon}, \\
U_0 := \setcond{x\in \X_0}{\norm{x}{\X_0}\leq C\varepsilon}, \\
V:= \setcond{y\in \Y}{\norm{y}{\Y}< \delta},
\end{align*}
with $C>0$ a constant satisfying $\norm{x}{\X_0}\leq C\norm{x}{\X_1}$, thus $U_1\subset U_0$.

Step 1: Show that for all $y\in V$, the function 
\begin{equation*}
\Phi_y(x) := x - \partial_x F(0,0)^{-1} F(x,y) = \partial_x F(0,0)^{-1}\left(\partial_x F(0,0)x - F(x,y)\right)
\end{equation*}
is a contraction mapping from $\left( U_1, \norm{\cdot}{\X_0}\right)$ to itself, provided that $\varepsilon, \delta >0$ are sufficiently small. 

As the fundamental theorem of calculus holds on Banach spaces as well, we have for $F\in C^1(D_1, \Z_0)$ 
and for all $x_1, x_2\in D_1$ 
\begin{equation}
\label{eq:fund_calc_F}
F(x_1, y) - F(x_2,y) = \int_0^1 \partial_x F(x_2 + t(x_1 - x_2), y)(x_1 - x_2)\,\mathrm dt.
\end{equation}
Note that $\partial_x F(x_2 + t(x_1 - x_2), y) \in \L(\X_1, \Z_0)$ with extension in 
$\L(\X_0, \Z_0)$.

Now let $x_1, x_2\in U_1$, $y\in V$. Then $(x_j,y)\in D_1$ for $j=1,2$, and we use \eqref{eq:fund_calc_F} to arrive at
\[
\Phi_y(x_1) - \Phi_y(x_2) =
\partial_x F(0,0)^{-1} \left[
\int_0^1 \left( \partial_x F(0,0) - \partial_x F(x_2 + t(x_1 - x_2), y) \right)\,\mathrm dt (x_1 - x_2) 
\right].
\]
By choosing $\varepsilon, \delta >0$ smaller, if necessary, we get
\begin{align}
\label{eq:contraction}
\begin{aligned}
\norm{\Phi_y(x_1) - \Phi_y(x_2)}{\X_0} 
& \leq \norm{\partial_x F(0,0)^{-1}}{\L(\Z_0, \X_0)} \norm{x_1 - x_2}{\X_0}\\
& \quad \cdot \sup_{0\leq t\leq 1}\norm{\partial_x F(0,0) - \partial_x F(x_2 + t(x_1 - x_2), y)}{\L(\X_0, \Z_0)} 
 \\
& \leq \frac{1}{2} \norm{x_1 - x_2}{\X_0},
\end{aligned}
\end{align}
where the second inequality holds because of the condition $\partial_x F\in C(D_1,\L(\X_0, \Z_0))$ in assumption (II), which for sufficiently small $\varepsilon, \delta >0$ implies
\[
\sup_{0\leq t\leq 1}\norm{\partial_x F(0,0) - \partial_x F(x_2 + t(x_1 - x_2), y)}{\L(\X_0, \Z_0)} << 1.
\]
Next, we show that $\Phi_y(x)\in U_1$ for $(x,y)\in D_1$. 
We estimate $\norm{\Phi_y(x)}{\X_1}$ for $(x,y)\in D_1$: Choosing $\delta = \delta(\varepsilon)>0$ smaller, if necessary, we obtain 
\begin{align*}
\norm{\Phi_y(x)}{\X_1} 
& \leq \norm{\partial_x F(0,0)^{-1}}{\L(\Z_1, \X_1)}\norm{\partial_x F(0,0)x - F(x,y)}{\Z_1} \\
& \leq \norm{\partial_x F(0,0)^{-1}}{\L(\Z_1, \X_1)}\norm{\partial_y F(0,0)}{\L(\Y,\Z_1)}\delta + o(\varepsilon + \delta(\varepsilon)) \\
&  \leq \varepsilon.
\end{align*}

Step 2: Construct a mapping $u\colon U_0\to V$. 

Let $y\in V$ arbitrary but fixed. 
The inductively defined sequence $\left(x_j\right)_{j\in\N_0}$ with $x_0:=0$, $x_{j+1} := \Phi_y(x_j)\in U_1\subset U_0$ for $j\in\N$ is a Cauchy sequence and thus converges in $\norm{\cdot}{\X_0}$ to some $x_\infty\in U_0$. 
Because $F\in C(D_0, \Z_0)$, this
implies
\[
\norm{F(x_\infty, y)}{\Z_0} = \lim_{j\to\infty} \norm{F(x_j, y)}{\Z_0} 
\leq \lim_{j\to\infty} \norm{\partial_x F(0,0)}{\L(\X_0,\Z_0)} \norm{x_j - x_{j+1}}{\X_0} =0,
\]
where we used $F(x_j, y) = \partial_x F(0,0) (x_j - x_{j+1})$ for $j\in \N_0$ by definition of $\Phi_y(x)$.
We set $u(y) := x_\infty \in U_0$ for $y\in V$. 

Step 3: Show (ii) of the theorem. 

If $x_1, x_2\in U_1$ and $y\in V$ with $F(x_j, y)=0$ for $j=1,2$, then
\[
\norm{x_1 - x_{2}}{\X_0} = \norm{\Phi_y(x_1) - \Phi_y(x_2)}{\X_0} 
\overset{\eqref{eq:contraction}}{\leq} \frac{1}{2} \norm{x_1 - x_{2}}{\X_0},
\]
and therefore $x_1 = x_2$.
\end{proof}

\begin{theorem}\label{thm:ift_part2}
Assume the following. 
\begin{enumerate}[label=(\Roman*)]
	\item 
	Consider Banach spaces $\X_2 \hookrightarrow \X_1 \hookrightarrow \X_0$, 
	$\Z_2 \hookrightarrow \Z_1 \hookrightarrow \Z_0$
	and $\Y$.
	Let $D_2\subset D_1\subset D_0$ be open sets such that $(0,0)\in D_j \subset \X_j\times \Y$ for $j=0,1,2$.
	\label{thm:ift_part2_a1}
	\item For $j=1,2$, let $F\in C^1(D_j, \Z_{j-1})\cap C(D_{j-1}, \Z_{j-1})$ with $F(0,0)=0$ and further
	$\partial_x F\in C(D_j, \L(\X_{j-1}, \Z_{j-1}))$. 
	This is to be understood such that for $(x,y)\in D_j$, the partial deriative 
	$\partial_x F(x,y)\in \L(\X_j, \Z_{j-1})$ can be extended to 	
	$\overline{\partial_x F}(x,y)\in \L(\X_{j-1}, \Z_{j-1})$ and 
	$\overline{\partial_x F} \in C(D_j, \L(\X_{j-1}, \Z_{j-1}))$.
	\label{thm:ift_part2_a2}
	\item For $j=1,2$, the mapping $F\colon D_j\to \Z_j$ is Fr\'echet-differentiable at $(0,0)$.
	\label{thm:ift_part2_a3}
	\item For $j=0,1,2$, the inverse $\partial_x F(0,0)^{-1} \in \L(\Z_j, \X_j)$ exists. 
	\label{thm:ift_part2_a4}
\end{enumerate}
Then there exist neighbourhoods of zero $0\in U_0\subset \X_0$, $0\in U_1\subset \X_1$ and $0\in V\subset \Y$ such that there is a function 
$u\colon V\to U_1$ satisfying 
\begin{enumerate}[label=(\roman*)]
	\item $F(u(y),y)=0$ for all $y\in V$, $u(0)=0$, 
	\item if $x\in U_1$, $y\in V$ such that $F(x,y)=0$, then $x=u(y)$, and
	\item it holds $u\in C^1(V, \X_0)$, and 
	\begin{equation*}
	u'(y) = -\partial_x F(u(y),y)^{-1}\partial_y F(u(y),y),
	\end{equation*}
	with $\partial_x F(u(y),y)^{-1} \in \L(\Z_0, \X_0)$ and $\partial_y F(u(y),y)\in \L(\Y, \Z_0)$.
\end{enumerate}
\end{theorem}

\begin{proof}
Step 1: Existence and uniqueness of $u$.

Applying Theorem \ref{thm:ift_part1} twice, we arrive at the existence of neighbourhoods $0\in V\subset \Y$, $0\in U_j\subset \X_j$, $j=0,1$, and at the existence of a mapping $u\colon V\to U_1$
such that 
\begin{itemize}[label={$\circ$}]
	\item $F(u(y), y)=0$ for $y\in V$, $u(0)=0$, and   
	\item for $x_1, x_2\in U_1$, $y\in V$ with $F(x_i, y) = 0$, $i=1,2$, we have $x_1 = x_2$. 
\end{itemize}
Thus, for $x_1\in U_1$ and $y\in V$ such that $F(x,y)=0$, we have $x=u(y)$. This shows (i) and $(ii)$ of Theorem \ref{thm:ift_part2}.

Step 2: Show Lipschitz-continuity of $u$, i.e.\ $u\in C^{0,1}(V,U_0)$.

Consider $y_1, y_2\in V$. Then with $u$ as above, i.e.\ $u(y_i)\in U_1$ for $i=1,2$, we have 
\begin{equation*}
\begin{aligned}
\norm{u(y_1)-u(y_2)}{\X_0} & = \norm{\Phi_{y_1}(u(y_1)) - \Phi_{y_2}(u(y_2))}{\X_0} \\ 
& \leq \norm{\Phi_{y_1}(u(y_1)) - \Phi_{y_1}(u(y_2))}{X_0} + \norm{\Phi_{y_1}(u(y_2)) - \Phi_{y_2}(u(y_2))}{\X_0} \\
& \overset{\eqref{eq:contraction}}{\leq} \frac{1}{2} \norm{u(y_1-u(y_2))}{\X_0} \\
	&  \quad + \norm{\partial_x F(0,0)^{-1}}{\L(\Z_0, \X_0)}\norm{F(u(y_2),y_1) - F(u(y_2),y_2)}{\Z_0}.
\end{aligned}
\end{equation*}
This implies 
\begin{equation*}
\begin{aligned}
\norm{u(y_1)-u(y_2)}{\X_0} & \leq 2\norm{\partial_x F(0,0)^{-1}}{\L(\Z_0, \X_0)} \\
& \quad \cdot\sup_{0\leq t\leq 1}\norm{\partial_y F(u(y_2),y_2 + t(y_1 -y_2))}{\L(\Y,\Z_0)}\norm{y_1-y_2}{\Y} \\
& \leq L \norm{y_1-y_2}{\Y}, 
\end{aligned}
\end{equation*}
because the $\sup$-term is uniformly bounded in $V$ for sufficiently small $\varepsilon, \delta >0$ defining the sets as in the proof of Theorem \ref{thm:ift_part1}, since $\partial_y F$ is continuous around $(0,0)$.

Step 3: Show $u\in C^1(V, \X_0)$.

We have for $y, h\in V$ s.th. $y+h\in V$ 
\begin{align*}
0 & = F(u(y+h), y+h) - F(u(y), y) \\
& = F(u(y+h), y+h) - F(u(y+h), y) + F(u(y+h), y) - F(u(y), y) \\
& = \partial_y F(u(y+h),y)[h] + o_{Z_0}(\norm{h}{\Y}) + F(u(y+h), y) - F(u(y), y) \\
& = \partial_y F(u(y),y)[h] + \left(\partial_y F(u(y+h),y) + \partial_y F(u(y),y) \right)[h] + o_{Z_0}(\norm{h}{\Y})\\ 
		& \quad + F(u(y+h), y) - F(u(y), y) \\
& = \partial_y F(u(y),y)[h] + o_{Z_0}(\norm{h}{\Y}) \\
		& \quad + \int_0^1 \partial_x F(u(y) + t(u(y+h) - u(y)),y) \,\mathrm dt [u(y+h) - u(y)],
\end{align*}
and further 
\begin{align*}
& \int_0^1 \partial_x F(u(y) + t(u(y+h) - u(y)),y) \,\mathrm dt [u(y+h) - u(y)] \\
& = \partial_x F(u(y),y)[u(y+h) - u(y)] + o_{\Z_0}(\norm{h}{\Y}).
\end{align*}
Therefore, 
\[
u(y+h) - u(y) = \partial_x F(u(y),y)^{-1} \partial_y F(u(y),y)[h] + o_{\X_0}(\norm{h}{\Y}), 
\]
yielding $u\in C^1(V, \X_0)$ and (iii).

Note that $\partial_x F(u(y),y) \in \L(\X_0, \Z_0)$ is invertible for $y\in V$, since 
$\partial_x F(0,0)$ is invertible and $\norm{u(y)}{\X_1} + \norm{y}{\Y} << 1$.
\end{proof}

%__________________applIFT________________________%
\section{Proof of Theorem \ref{thm:main}}
\label{sect:existence_stability_ift}

With the tool of the modified implicit function theorem, Theorem \ref{thm:ift_part2}, at hand, we are now able to prove Theorem \ref{thm:main}, that is, the existence and uniqueness of admissible sets $\Omega_\rho$ with barycenter zero that solve the perturbed overdetermined problem \eqref{eq:serrin_perturbed_rho}, as well as a stability estimate. 

\begin{remark}
Considering the somewhat unintuitive partial derivative $\partial_{\rho_1, \rho_2, g_2}$ in Lemma \ref{lem:regularity_g_zero} was necessary to arrive at bijectivity and to be able to apply Theorem \ref{thm:ift_part2}. The partial derivative $\partial_\rho G(0)$ is not bijective. 

In addition to that, keeping in mind the nature of the problem discussed in Section \ref{sect:degeneracy}, 
the set $\Omega_\rho$ will only depend on the perturbations that do not induce a mere translation of the problem. 
$\rho$ depending on $g_1$ (instead of $g$) is a consequence of that setting.
\end{remark}

\begin{proof}[{Proof of Theorem \ref{thm:main}}]
We confirm the requirements for Theorem \ref{thm:ift_part2}. 
For \ref{thm:ift_part2_a1}, we set 
\begin{align*}
\X_{j} & = h^{j+2+\alpha}(\S)\times X_2^\bot = h^{j+2+\alpha}(\S)\times K \text{ and } \\
\Z_{j} & = X_{j+1}\times \R^n \times X_{2}^\bot = X_{j+1}\times \R^n \times K 
\end{align*}
for $j=0,1,2$, $\Y  = X_2$, and $D_j$ accordingly. 
By Lemma \ref{lem:regularity_g}, \eqref{eq:regularity_g_general} and \eqref{eq:regularity_g_derivative_extension}, \ref{thm:ift_part2_a2} is satisfied.
Lemma \ref{lem:regularity_g_zero} implies \ref{thm:ift_part2_a3} and \ref{thm:ift_part2_a4}. 

Thus, Theorem \ref{thm:ift_part2} implies the existence of a neighbourhood 
$0\in V\subset X_2$ 
as well as neighbourhoods
$0\in U_j \subset h^{j+\alpha}(\S)\times X_2^\bot = h^{j+\alpha}(\S)\times K$, $j=2,3$, 
such that there is a function 
$(\rho, g_2) \colon V \to U_3$ with  
$G(\rho(g_1), g_1 + g_2(g_1))=0 $ for all $g_1\in V$ and $(\rho, g_2)(0)=0$.
 
Furthermore, $(\rho, g_2)$ is unique in $U_3$ and we have $(\rho, g_2) \in C^1(V, U_2)$.
Differentiating 
{$G(\rho(g_1), g_1 + g_2(g_1))=0$} 
with respect to $g_1$ and evaluating it at $g_1=0$ in direction $\tilde g_1$, we get 
\begin{align*}
0 & = D_{g_1} G \left(\rho_1(g_1), \rho_2(g_1), g_1 + g_2(g_1)\right)\big|_{(0)} [\tilde g_1] \\
& = D_{\rho_1, \rho_2, g_2} G (0) \left[\partial_{g_1} \rho_1(0)[\tilde g_1], \partial_{g_1} \rho_2(0)[\tilde g_1], \partial_{g_1} g_2(0)[\tilde g_1] \right] + \partial_{g_1} G(0) [\tilde g_1] \\
& = 
\begin{pmatrix}
	\partial_{\rho_1} F(0) [\partial_{g_1} \rho_1(0)[\tilde g_1]]  \\
	\int_{\S} \sigma_1 \partial_{g_1} \rho_2(0)[\tilde g_1] \,\mathrm d\sigma  \\
	\vdots \\
	\int_{\S} \sigma_n \partial_{g_1} \rho_2(0)[\tilde g_1] \,\mathrm d\sigma \\
		\partial_{g_1} g_2(0)[\tilde g_1] 
\end{pmatrix}  
+ 
\begin{pmatrix}
	\tilde g_1 \\ 0 \\ 0
\end{pmatrix}.
\end{align*}
This yields 
\begin{align*}
\begin{aligned}
\partial_{g_1} \rho_1(0)[\tilde g_1] & = \partial_{\rho_1} F(0)^{-1}[\tilde g_1], \\
\partial_{g_1} \rho_2(0)[\tilde g_1] & = 0, \text{ and} \\
\partial_{g_1} g_2(0)[\tilde g_1] & = 0, 
\end{aligned}
\end{align*}
where the last equation results from $\partial_{g_1} \rho_2(0)[\tilde g_1]\in X_2^{\bot} = K$,
and we arrive at the stability estimates in \eqref{eq:stability_ift}. 
\end{proof}

%______________________BIB_________________%
 \bibliographystyle{plainnat}
 \bibliography{stability_serrin_ift_arxiv}{}

\end{document}